\documentclass[joc]{ipart}

\usepackage{amsmath,amssymb,amsthm,graphicx,verbatim,fullpage,bbm,nicefrac}

\pubyear{2014}
\volume{0}
\issue{0}
\firstpage{1}
\lastpage{1}

\startlocaldefs

\newtheorem{theorem}{Theorem}
\newtheorem*{cheeger}{Cheeger Inequality}
\newtheorem{cor}[theorem]{Corollary}
\newtheorem{lemma}[theorem]{Lemma}

\newtheorem{obs}[theorem]{Observation}
\theoremstyle{remark}

\newcommand{\one}{\ensuremath{\mathbbm{1}}}

\DeclareMathOperator{\BIN}{Bin}
\newcommand{\bin}[1]{\ensuremath{\BIN\paren{#1}}}
\newcommand{\e}{{\mathbb{E}}}

\newcommand{\p}{{\mathbb{P}}}
\DeclareMathOperator{\Vol}{Vol}
\newcommand{\vol}[1]{\ensuremath{\Vol\!\paren{#1}}}

\newcommand{\eps}{\epsilon}
\newcommand{\ba}{\backslash}
\newcommand{\prob}[1]{\ensuremath{\p\!\left( #1\right)}}
\newcommand{\expect}[1]{\ensuremath{\e\!\left[ #1 \right]}}
\newcommand{\norm}[2][]{\ensuremath{\left\| #2 \right\|_{#1}}}

\newcommand{\R}{\ensuremath{\mathbb{R}}}

\newcommand{\paren}[1]{\ensuremath{\left( #1 \right)}}
\newcommand{\set}[1]{\ensuremath{\left\{ #1 \right\}}}
\newcommand{\size}[1]{\ensuremath{\left| #1 \right|}}
\newcommand{\abs}[1]{\ensuremath{\left| #1 \right|}}
\newcommand{\ceil}[1]{\ensuremath{\left\lceil #1 \right\rceil}}

\newcommand{\half}{\ensuremath{\nicefrac{1}{2}}}

\renewcommand{\L}{\ensuremath{\mathcal{L}}}
\newcommand{\Laplace}[1]{\ensuremath{\L\paren{#1}}}

\newcommand{\bigTheta}[1]{\ensuremath{\Theta\!\paren{#1}}}

\newcommand{\skg}{$t^{\textrm{th}}$-order stochastic Kronecker graph}

\newcommand{\lilOh}[1]{\ensuremath{\mathit{o}\!\paren{#1}}}
\newcommand{\bigOh}[1]{\ensuremath{\mathcal{O}\!\left( #1 \right)}}
\newcommand{\bigOmega}[1]{\ensuremath{\Omega\!\left( #1\right)}}
\newcommand{\inner}[2]{\ensuremath{\left\langle #1, #2 \right\rangle}}
\newcommand{\ER}{Erd\H{o}s-R\'enyi}

\endlocaldefs

\begin{document}

\begin{frontmatter}

\title{Connectivity and Giant Component of Stochastic Kronecker Graphs}
\runtitle{Connectivity and Giant Component of SKGs}

\begin{aug}
\author{\fnms{Mary}
  \snm{Radcliffe}\corref{}\ead[label=e1]{radcliffe@math.washington.edu}\ead[label=u1,url]{www.math.washington.edu/$\sim$maryr25}}
\address{University of Washington -- Seattle \\ \printead{e1} \\ \printead{u1}}
\and
\author{\fnms{Stephen~J.} \snm{Young}\ead[label=e2]{stephen.young@louisville.edu}\ead[label=u2,url]{www.math.louisville.edu/$\sim$syoung}}
\address{University of Louisville \\ \printead{e2} \\ \printead{u2}} 

\runauthor{M.~Radcliffe \&  S.~J.~Young}
\affiliation{University of Washington and University of Louisville}
\end{aug}

\begin{abstract}
Stochastic Kronecker graphs are a model for complex networks
where each edge is present independently according to the Kronecker
(tensor) product of a fixed matrix $P \in [0,1]^{k \times k}$.  We
develop a novel correspondence between the adjacencies in a general
stochastic Kronecker graph and the action of a fixed Markov chain.
Using this correspondence we are able to generalize the arguments of
Horn and Radcliffe on the emergence of the giant component from the
case where $k = 2$ to arbitrary $k$.  We are also able to use this
correspondence to completely analyze the connectivity of a general
stochastic Kronecker graph.
\end{abstract}



\end{frontmatter}

\section{Introduction}
In many ways the study of random graphs traces its history back to the
seminal work of Erd\H{o}s and R\'{e}nyi showing that there exists a
rapid transition between the regimes of a random graph consisting of many small
components, a random graph having one ``giant'' component, and a
random graph being connected~\cite{Erdos:RandomGraphs}.  Because of
their central role in the history of random graphs these phase
transitions have been extensively studied, see for instance
\cite{Bollobas:DiameterER,Bollobas:evolutionER,Bollobas:NormalGC,Bollobas:BranchingGC,Lubetzky:YoungGC,Spencer:AchlioptasGC,Luczak:CriticalComponents},
among numerous others.  We contribute to this ongoing discussion by
providing a sharp transition for the emergence of both the giant
component and connectivity for the stochastic Kronecker graph,
a generalization of the standard Erd\H{o}s-R\'{e}nyi binomial random
graph model, $\mathcal{G}(n,p)$.

More formally, recall that the Kronecker or tensor product of two matrices $A \in
\R^{m \times n}$ and $B \in \R^{p \times q}$ is a matrix $A \otimes B =
C \in \R^{mp \times nq}$.  For $i \in [m], j \in [n], s \in [p],$ and
$t \in [q]$ the entry  $C_{(i-1)m + s, (j-1)n + t}$ is $A_{ij}B_{st}$,
that is 
\[ A \otimes B =  C = \left[\begin{matrix} A_{1,1}B & A_{1,2}B & \cdots
    & A_{1,n}B \\ A_{2,1}B & A_{2,2}B & \cdots & A_{2,n}B \\ \hdots &
    \hdots & \ddots & \hdots \\ A_{m,1}B & A_{m,2}B & \cdots &
    A_{m,n}B \end{matrix}\right].\]

Letting $P \in [0,1]^{k \times k}$ be a symmetric matrix, the \skg\
generated by $P$ is formed by taking the $t$-fold Kronecker product of
$P$, denoted $P^{\otimes t}$, and using this as the probability matrix
for a graph with independent edges.  That is, each edge $\set{i,j}$ is
present independently with probability $P^{\otimes t}_{ij} = P^{\otimes t}_{ji}$.

The stochastic Kronecker graph was originally proposed as a model for
the network structure of the internet with the property that it could
be easily fit to real world data, especially in the case where the generating matrix
was $\left[ \begin{matrix} \alpha & \beta \\ \beta &
    \gamma \end{matrix}\right]$ where $0 < \gamma \leq \beta \leq
\alpha < 1$~\cite{Leskovec:KroneckerGeneration}.  As such, there have been
several papers analyzing structural properties of the stochastic
Kronecker graph when the generating matrix is a $2 \times 2$
matrix~\cite{Leskovec:KroneckerGeneration,Leskovec:Kronecker,Mahdian:Kronecker,
  Radcliffe:KroneckerGiant}.  Most relevant to this
current work are the results of Mahdian and Xu~\cite{Mahdian:Kronecker} who
anaylzed the connectivity, diameter, and the emergence of the giant
component with $0 < \gamma \leq \beta \leq \alpha < 1$, and the work of the first author and Horn
who analyzed the emergence and size of the giant component for arbitrary
$\alpha,\beta,\gamma \in (0,1)$~\cite{Radcliffe:KroneckerGiant}. 
In this work
we consider the case of an arbitrarily sized generating matrix, and develop necessary and sufficient conditions for the emergence of the giant component and connectivity. The key tool to analyzing these graphs is to tie the structure of the graph to a fixed Markov chain on the underlying generating matrix. Using this underlying structure, one can analyze the graph structure more completely than with traditional tools.

Given a $t^{\textrm{th}}$-order stochastic Kronecker graph with
generating matrix $P$, we define $W(P)$ to be the weighted graph on
$[k]$, where weights are as given in $P$.  We will occasionally refer
to $W$ as the underlying graph of $G$.  We also define the backbone graph
of the matrix $P$, $B(P)$, as the subgraph of $W(P)$ consisting of the
edges assigned weight 1.  That is, $B(P)$ is a graph on the vertices
$[k]$ where $\set{i,j}$ is an edge if and only if $P_{ij} = P_{ji} =
1$.  When the matrix $P$ is clear, we will neglect the dependence on
$P$ and write simply $W$ and $B$.

Our primary results can be summarized as follows.
\begin{theorem}\label{T:Master}
Let $G$ be \skg\ generated by a symmetric matrix $P \in [0,1]^{k \times k}$ which has column sums $c_1\leq c_2\leq\dots\leq
c_k$.  Let $n = k^t$ be the number of vertices of $G$.  
\begin{enumerate}
\item If $W$ is disconnected or bipartite, then the largest component of $G$ has
  size $\bigOh{(k-1)^t} \in \lilOh{n}.$ \label{M:disc_bip}
\item If $W$ is connected and non-bipartite and $\prod_i c_i < 1$,
  then there is some $0 < \alpha < 1$ such that with probability at
  least $1 - e^{-\bigTheta{n^{\alpha}}}$ there are at least $n -
  \bigOh{n^{\alpha}}$ isolated vertices in $G$. \label{M:small}
\item If $W$ is connected, non-bipartite, $\prod_i c_i = 1$, and the
  $c_i$'s are not identically one, then there is a positive constant $\alpha$
  such that with probability at least $1 - e^{-\bigTheta{n^{\alpha}}}$,
  the largest component of $G$ has size $\bigTheta{n}$, that is, $G$
  has a giant component. \label{M:giant_bound}
\item If $W$ is connected,  non-bipartite, and $\prod_i c_i > 1$, then there is a positive constant $\alpha$ such that with
  probability at least $1 - e^{-\bigTheta{n^{\alpha}}}$ 
   the largest component of $G$ has size
  $\bigTheta{n}$. \label{M:giant}  
\item If $W$ is connected, non-bipartite, and $c_1 < 1$, then there is
  a positive constant $\alpha$ such that $G$ has at least
  $\ln(n)^{(1-\lilOh{1})\ln\ln(n)}$ isolated vertices with probability
  at least $1 - \bigOh{n^{-\alpha}}$. \label{M:discon_c1}
\item If $W$ is connected, non-bipartite, $c_1 = 1$, and $B$ has a
  vertex of degree zero, then there is some positive constant $\alpha$
  such that $G$ has at least $\ln(n)^{(1-\lilOh{1})\ln\ln\ln(n)}$ isolated vertices with probability at least $1 - \bigOh{n^{-\alpha}}$.\label{M:discon}
\item If $W$ is connected and non-bipartite, $c_1 = 1$, and $B$ has no
  vertices of degree zero, then there is a constant $\alpha > 0$ such
  that $G$ is connected with probability at least  $1 - e^{-(1-\lilOh{1})n^{\alpha}} $. \label{M:connect_bound}
\item If $W$ is connected and non-bipartite and $c_1 > 1$, then there
  is a constant $\alpha > 0$ such that $G$ is
  connected with probability at least $1-e^{-(1-\lilOh{1})n^{\alpha}}$.   \label{M:connected}
\end{enumerate}
\end{theorem}

We note that item (\ref{M:connected}) above is typical for the emergence
of connectivity; that is, the graph is connected asymptotically almost
surely precisely when asymptotically almost surely the minimum
degree is at least 1.  In fact, taking (\ref{M:discon_c1}), (\ref{M:discon}), (\ref{M:connect_bound}),
and (\ref{M:connected}) together we can see that a stochastic Kronecker
graph is connected precisely when the minimum degree is at least 1
asymptotically almost surely.  From this viewpoint, the slightly unnatural seeming
condition on the backbone graph $B$ is simply the condition needed to
assure that $G$ has no isolated vertices.   

The folklore in the study of random graphs asserts that, in general,
the giant component should emerge when the average expected degree is 1,
see for instance \cite{Bollobas:evolutionER,Chung:ComponentsExpectDeg,Erdos:RandomGraphs,Martin:GiantCompSubgraph}.  As the average expected degree in a \skg\
is $k^{-t}\paren{c_1 + \cdots + c_k}^t$, this suggests that the
transition occurs when $\frac{1}{k} \paren{c_1 + \cdots + c_k} > 1$.
However, as parts (\ref{M:small}) and (\ref{M:giant}) of Theorem
\ref{T:Master} show, the transition actually occurs when
$\paren{\prod_i c_i}^{\frac{1}{k}} > 1$.  Noting that the expected
degrees in stochastic Kronecker graphs follow a multinomial
distribution (see Section \ref{defsec}), this condition can be seen as
equivalent (asymptotically) to the condition that \emph{median}
expected degree is at least one.  Thus our results may suggest that
the average expected degree is not as deeply connected to giant
component as previously thought, because in many of the standard
random graph models, such as the \ER\ random graph, the average and the
median expected degree agree. That is, it may be that the median is truly the
determining factor for such structures.  It is also worth noting that
Spencer has conjectured based in part on \cite{Horn:GiantCompSubgraph,Horn:PercolationGC}, that the correct intuition is
that the emergence of the giant component is tied to the second order
average degree~\cite{Spencer:GCOverview}. 

To prove Theorem \ref{T:Master}, we will develop several general
results on $G$, and then apply these results to the specific situations above. In
particular, we are able to tie the adjacency structure of $G$ to a finite state
Markov chain on $W$. Using this association, we can take advantage of
the finite structure of $W$ to build theory regarding the asymptotically growing
structure $G$. 

\section{Definitions and Tools}\label{defsec}

Given a stochastic Kronecker graph $G$ generated by $P$, let $A$ be
the adjacency matrix of $G$ and $D$ the diagonal matrix of degrees in
$G$. Let $c_1\leq c_2\leq \dots\leq c_k$ be the column sums of $P$ (note that we can assume these are nondecreasing without loss of generality), and let $C$ be the diagonal matrix of column sums in $P$. 

 We note that there are multiple means of describing
the entries of
the probability matrix $P^{\otimes t}$ to take advantage of the
Kronecker product structure.  One point of view that is particularly
helpful is to define  a
bijection $w \colon V(G)\to [k]^t$, so that each vertex of $G$ is represented by a
word of length $t$ in $[k]$. We will often identify the vertex to its
corresponding word, and write $v=(v_1, v_2, \dots, v_t)$. Given an
appropriate choice of bijection,  for any two vertices $u$ and $v$,
the probability that $u$ and $v$ are adjacent is \[ p_{uv} =
\prod_{i=1}^t P_{u_iv_i}.\] That is to say, we take the product of
entries of the generating matrix $P$, where entries correspond to the
pairs of components in the words representing $u$ and $v$. We will sometimes use the
notation $P^{\otimes t}_{u, v}$ to refer to the $w(u), w(v)$ position
in $P^{\otimes t}$, where we index the matrix by the ordered words
obtained via the Kronecker product, and we note that $p_{uv}=P^{\otimes t}_{u, v}$.  

We shall use the notation $u\sim v$ to indicate that $u$ is adjacent to $v$. When ambiguous, we write $u\sim_Gv$ to indicate that $u$ is adjacent to $v$ in the graph $G$.

Now, suppose that $w(v)$ has $a_1$ coordinates equal to 1, $a_2$ coordinates equal to 2, and so on. It is straightforward to calculate that
\[\expect{\deg(v)} = c_1^{a_1}c_2^{a_2}\dots c_k^{a_k}.\]
From this we can see that the stochastic Kronecker graph is defined precisely so that the
expected adjacency matrix $\bar{A}=P^{\otimes t}$, and the expected
degree matrix $\bar{D}=C^{\otimes t}$. At times we will wish to
emphasize the graph structure of $P^{\otimes t}$, and thus will use $W^{\otimes t}$ to refer to the weighted complete graph with weights given by $P^{\otimes t}$.  

Moreover, it will frequently be of interest to know the number of
coordinates in $w(v)$ equal to each symbol in $[k]$. To that end, we
define the {\it signature} of $v$ to be $\sigma(v)=(\sigma_1,
\sigma_2, \cdots, \sigma_k)$, where $\sigma_i$ is the proportion of
symbols in $w(v)$ equal to $i$. For example, if $k=5$ and
$w(v)=121251$, we would have $\sigma(v)=(\frac{1}{2}, \frac{1}{3}, 0,
0, \frac{1}{6})$. We will denote  by
$\mathcal{S} = \set{\paren{\sigma_1, \ldots,\sigma_k} \mid \sigma_i
  \geq 0, \sum_i \sigma_i = 1}$ the space of possible signatures. 
 Often we will establish an underlying signature for a vertex and then take $t$ to infinity; this will generally result in noninteger values for the number of letters of a particular value in $w(v)$. This can be overlooked, however, as rounding to the next integer appropriately will not change the asymptotic features of the vertices, and so we will often assume that a vertex can take any signature. 

Let $L = \paren{\ln(c_1), \ln(c_2), \cdots, \ln(c_k)}$. We will make frequent use of the simple observation that 
\[\ln\paren{\expect{\deg(v)}} = t \inner{\sigma(v)}{ L},\]
where $\inner{\cdot}{\cdot}$ represents the standard dot product.

\subsection{Markov chains in $G$ and $W$}\label{MCdescription}

Let $W^{\otimes t}$ be the weighted complete graph on $V(G)$, with the
weight of edge $uv$ equal to $P^{\otimes t}_{u,v}$. Let $v$ be a
vertex in $W^{\otimes t}$ with signature $\sigma=(\sigma_1, \sigma_2,
\dots, \sigma_k)$. Define $Z^{(v)}$ to be a random variable that takes
values in $\mathcal{S}$, where $Z^{(v)}$ is the signature of a
randomly chosen neighbor of $v$ according to the probability
distribution defined by the weights of the edges. That
is, \[\p(Z^{(v)}=\tau)=\sum_{\sigma(u) = \tau}\frac{P^{\otimes t}_{u, v}}{\deg_{W^{\otimes t}}(v)}.\]
That is to say, $Z^{(v)}$ is the signature of the vertex obtained
after taking one step in the uniform random walk on $W^{\otimes t}$.

For each $i\in [k]$, let $X^{(i)}$ be the random variable that takes values in $[k]$, with $\p(X^{(i)}=j)=\frac{P_{ij}}{c_i}$. Note that for $v=(v_1, v_2, \dots, v_t)$ fixed, we have 
\[
\p(X^{(v_1)}\times X^{(v_2)}\times \dots \times X^{(v_t)} = (u_1, u_2, \dots, u_t)) = \prod_{i=1}^t \frac{P_{v_iu_i}}{c_{v_i}}
= \frac{P^{\otimes t}_{u, v}}{\deg_{W^{\otimes t}}(v)} .
\]
Thus we can consider $Z^{(v)}$ as giving the signature of a randomly
chosen neighbor of $v$, chosen according to the product distribution
$X^{(v_1)}\times X^{(v_2)}\times \dots \times X^{(v_t)}$. As the
signature is independent of order, for the purposes of analyzing
$Z^{(v)}$, we may write this distribution as
$(X^{(1)})^{\sigma_1t}\times(X^{(2)})^{\sigma_2t}\times\dots\times(X^{(k)})^{\sigma_kt}$. Therefore,
for all $i\in [k]$, letting $Z_i^{(v)}$ be the $i^{\textrm{th}}$
component of the signature $Z^{(v)}$, we have
\[
\expect{Z^{(v)}_i}  = \frac{1}{t}\sum_{j=1}^k (\sigma_jt) \p(X^{(j)}=i)= \sum_{j=1}^k \sigma_j \frac{P_{ij}}{c_j}.
\]

On the other hand, let $M=C^{-1}P$, the transition probability matrix
for the uniform random walk on  $W$ and notice that the matrix product $\sigma M$ has $i^{\textrm{th}}$ coordinate
\[
\paren{(\sigma_1, \sigma_2, \dots, \sigma_k)M}_i = \sum_{j=1}^k \sigma_j M_{ij}
= \sum_{j=1}^k \sigma_j \frac{P_{ij}}{c_j}
= \expect{Z^{(v)}_i}
\] Thus, $\sigma M = \expect{Z^{(v)}}$.

Therefore, we can think of the distribution of a random walk on $W$ as
the expected signature of a vertex in a random walk on $W^{\otimes
  t}$. Let $\pi=(\pi_1, \pi_2, \dots, \pi_k)$ be the stationary
distribution of the random walk on $W$, so $\pi M= \pi$. It is a
simple exercise to verify that $\pi_i = \frac{c_i}{\sum_j c_j}$. We
will show in Section \ref{Structural} that the collection of
signatures close to $\pi$ will in fact, asymptotically almost surely,
form a connected subgraph in $G$, and further, by leveraging the convergence
of the Markov chain on $W$, we can assure a giant component.

\subsection{Tools and Notation}\label{S:Tools}

For a given graph $G$, the normalized Laplacian matrix for $G$ is the matrix $\Laplace{G}=I-D^{-\half}AD^{-\half}$.  We denote the eigenvalues of $\Laplace{G}$ by $0=\lambda_0\leq \lambda_1\leq\dots\leq\lambda_{n-1}$. If there is any ambiguity, we write $\lambda_i(\Laplace{G})$ to specify that the eigenvalues are from the normalized Laplacian, and more generally $\lambda_i(M)$ to denote the $i^{\textrm{th}}$ smallest eigenvalue of a Hermitian matrix $M$. We sometimes refer to these as the Laplacian eigenvalues of $G$. We shall use the following standard facts from spectral graph theory.

\begin{theorem}[see, for example, \cite{Chung:Spectral}]\label{T:standards}
Let $G$ be a graph with Laplacian eigenvalues $0=\lambda_0\leq \lambda_1\leq \dots\leq\lambda_{n-1}$. Then
\begin{enumerate}
\item $G$ is connected if and only if $\lambda_1>0$.
\item If $G$ is connected, then the diameter $D(G)$ of $G$ satisfies $D(G)\leq \ceil{\frac{\ln(n-1)}{\ln(1/(1-\lambda_1))}}$.
\item Let $D^{-1}A$ denote the probability transition matrix of a random walk on $G$. Then $\lambda$ is an eigenvalue of $\Laplace{G}$ with eigenvector $v$ if and only if $1-\lambda$ is an eigenvalue of $D^{-1}A$ with eigenvector $v$.
\end{enumerate}
\end{theorem}

Among our key tools will be the following theorem from Chung and the first author \cite{Radcliffe:Spectra} that gives spectral concentration in the normalized Laplacian of a general random graph.

\begin{theorem}[\cite{Radcliffe:Spectra}]\label{T:IndepSpec}
  Let $G$ be a random graph with independent edges generated according
  to the matrix $\mathcal{P}$.  Let $\mathcal{D}$ be the diagonal matrix of expected degrees and let $\delta$ denote the minimum expected
  degree. If  $\delta \geq 3\ln\paren{\frac{4n}{\epsilon}}$, then with probability
  at least $1-\epsilon$, for all $i$ \[ \abs{\lambda_i\paren{\Laplace{G}} -
    \lambda_i\paren{I - \mathcal{D}^{-\half}\mathcal{P}\mathcal{D}^{-\half}}} \leq
  3\sqrt{\frac{3\ln\paren{\frac{4n}{\epsilon}}}{\delta}}. \] 
\end{theorem}

We also make use of standard tools in spectral graph theory, chief among them the Cheeger inequality. For two sets $S, T$ of vertices in a graph $G$, define $e_G(S, T)$ to be the number of edges (or, in a weighted graph, the total weight of edges) for which one endpoint is in $S$ and the other in $T$. Note that an edge with both endpoints in $S\cap T$ is counted twice in this definition. Define $\Vol_G(S)=\sum_{v\in S} \deg(v)$. When the underlying graph is clear, we drop the subscript $G$ in the notation.

The Cheeger constant of a set $S$ with $\vol{S}\leq
\frac{1}{2}\vol{G}$ is defined to be $h(S)=\nicefrac{e(S, V\ba S)}{\vol{S}}$ and
Cheeger constant of $G$ is \[h_G=\min_{\substack{S\subset V\\
    \vol{S}\leq \frac{1}{2}\vol{G}}} h(S).\] The spectrum of a graph
is related to the Cheeger constant via the Cheeger Ineqaulity~\cite{Sinclair:MCMC,Jerrum:gap}.

\begin{cheeger}
For $G$ any graph, let $\lambda_1$ be the smallest nontrivial eigenvalue of $\L(G)$. Then
\[ \frac{1}{2}h_G^2\leq \lambda_1\leq 2h_G.\]
\end{cheeger}

As we will frequently be discussing Markov chains, we will pass
regularly between considering row vectors and column vectors. We will
always treat the signature of a vertex $v$ as a row vector, as well as
the vector $L$. The all-ones vector, $\one$, will be considered a row
vector as well. However, eigenvectors of a matrix are typically
assumed to be right eigenvectors, and are thus column vectors. Any other usages should be made clear by context.

In order to understand the rate of convergence of a Markov chain we
will use the relative pointwise distance. If $\pi$ is the
limiting distribution of the Markov chain, the relative pointwise distance of a
distribution $\sigma$ from $\pi$ is 
\[ \Delta_{RP}(\sigma) =  \max_i \frac{
  \abs{\sigma_i - \pi_i}}{\pi_i}.\]  
As we are interested in an overall rate of convergence we define
\[ \Delta(s) = \sup_{\sigma \in \mathcal{S}} \Delta_{RP}\paren{\sigma M^s}.\]
It is well known that the rate of decay of the relative pointwise distance can be controlled
by the spectral information of the Markov chain as given in the
following theorem, see for instance \cite{Chung:Spectral}.
\begin{theorem}\label{TVDist}
 Let $1=\lambda_0\geq\lambda_1\geq\dots\geq\lambda_{n-1}$ be the
 eigenvalues of the transition probability matrix of a uniform random
 walk on a connected, non-bipartite (weighted) graph $G$. Set $\lambda =
 \max\set{\abs{1-\lambda_1},{\abs{\lambda_{n-1}-1}}}$.  For any \[s>\frac{1}{\lambda}\ln\paren{\frac{\vol{G}}{\epsilon\delta_G}},\] we have $\Delta(s)<\epsilon$, where $\delta_G$ denotes the minimum degree in $G$.
\end{theorem}

The phrase \emph{asymptotically almost surely} in this paper will
always refer to asymptotics with respect to $t$, unless otherwise
noted. The norm $\| v \|$ will refer to the $\ell_{\infty}$-norm unless otherwise noted.

\section{Key Results}\label{Structural}

To prove the thresholds for connectivity and emergence of the giant
component in a stochastic Kronecker graph $G$ (Theorem \ref{T:Master},
items (\ref{M:giant}) and (\ref{M:connected})), we will use the
following structure. First, we show that $G$ contains a small set of
vertices that is connected asymptotically almost surely, in particular, those vertices that
are close to stationarity under the Markov chain described in Section
\ref{MCdescription}. We shall refer to this set as the ``connected
core'' of the graph. Although this will not be enough vertices to form a
giant component, we can then show that under certain conditions,
a positive fraction of the vertices in $G$ can be connected by a path to the
connected core. The thresholds given are precisely those conditions
needed to ensure that a positive fraction of the vertices exhibit this behavior. In retrospect, the arguments used by  Horn and the first author in
\cite{Radcliffe:KroneckerGiant} to show the emergence of the giant
component in the case where the generating matrix is $2\times
2$ can be viewed as a special case of our technique. Specifically, as the underlying Markov chain has only two states, the degree of each vertex is controlled by a single parameter, which significantly simplifies the argument. As a consequence, the authors in \cite{Radcliffe:KroneckerGiant} were able to analyze the giant component directly via counting techniques, without appealing to the underlying Markov chain.

In this section, we develop much of the underlying structure in $G$
via the random walk on $W$.  We begin with some elementary observations on the vertex degrees in $G$ and $W^{\otimes t}$.

\begin{lemma}\label{agoodneighbor}
Let $v$ be a vertex with signature $\sigma$ in a \skg\ $G$, such that $\inner{\sigma}{L}>0$. Let $d=e^{\inner{\sigma}{L}}$. For any $\delta>0$, we have
\begin{enumerate}
\item\label{concinGbar} $v$ has at least
  $d^t(1-2ke^{-2\delta^2t})$ neighbors
  in $W^{\otimes t}$ with signature $\tau$ such that
  $\norm{\tau-\expect{Z^{(v)}}}\leq \delta$.
\item\label{agoodneighborG} with probability at least $1-\exp(-\frac{d^{t}}{8}(1-2ke^{-2\delta^2t}))$, $v$ has at least $\frac{1}{2}d^t(1-2ke^{-2\delta^2t})$ neighbors in $G$ with signature $\tau$ such that $\norm{\tau-\expect{Z^{(v)}}}\leq \delta$. 
\end{enumerate}
\end{lemma}

\begin{proof}
  By the Hoeffding inequality, we have that for any $i$,
\[\p\left(t\abs{Z^{(v)}_i-\expect{Z^{(v)}_i}} >\delta t\right)\leq 2e^{-2\delta^2 t}\]
for any $\delta>0$. Therefore, by the union bound, we have
\[\p\left(\exists i\in[k]\hbox{ such that } t\abs{Z^{(v)}_i-\expect{Z^{(v)}_i}}>\delta t\right)\leq 2ke^{-2\delta^2t}.\]
This verifies item (\ref{concinGbar}).

For item (\ref{agoodneighborG}), note that by (\ref{concinGbar}), we
have that the expected number of neighbors of $v$ with
signature $\tau$ in the desired range is at least $d^t(1-2ke^{-2\delta^2t})$.  By Chernoff bounds, then, with probability at least
$1-\exp(-\frac{d^{t}}{8}(1-2ke^{-2\delta^2t}))$, we have at least $\frac{1}{2}d^t(1-2ke^{-2\delta^2t})$ neighbors with such a signature $\tau$.
\end{proof}

As an immediate corollary of this result we have the following.
\begin{cor}\label{C:tight} 
Let $v$ be a vertex with signature $\sigma$ in a \skg\ $G$, such that  $\inner{\sigma}{L} > 0$. Let $d =
e^{\inner{\sigma}{L}} > 1$.  With probability at least
$1-e^{-\frac{d^t}{12}}$, $v$ has at least $\frac{d^t}{3}$ neighbors
 $u$ with $\norm{\sigma(u) - \sigma M} \leq \sqrt{\frac{\ln(6k)}{2t}}$.
\end{cor}

Recall from Section \ref{MCdescription} that $\pi=(\pi_1, \pi_2, \dots,
\pi_k)$ is the stationary distribution of the random walk on $W$, with
$\pi_i = \frac{c_i}{\vol{W}}$ for all $i$. Given $\epsilon>0$, define
$\mathcal{S}_\epsilon = \{ v\in G \mid \forall i\in[k],
\sigma_i(v)>(1-\epsilon)\pi_i\}$. Notice that if $v\in S_\eps$ with signature $\sigma$, then we have, for all $i$, 
\[(\sigma M)_i = \sum_{j=1}^k \sigma_jM_{ij}\geq (1-\eps)\sum_{j=1}^k \pi_jM_{ij} = (1-\eps)\pi_i\]
by stationarity of $\pi$. Hence, if $v\in \mathcal{S}_\eps$, then $\expect{Z^{(v)}}$ is also in $\mathcal{S}_\eps$. We will show that under appropriate
conditions, this set of vertices $\mathcal{S}_\eps$ is connected asymptotically
almost surely, forming the small connected core described above.  To do this, we will show that vertices in $\mathcal{S}_\eps$ have exponentially large degree in $t$, and then use Theorem \ref{T:IndepSpec} to show the first eigenvalue in $\mathcal{S}_\eps$ is bounded away from zero. We first must address the degree of vertices in $\mathcal{S}_\eps$. To that end, we have the following Lemma:

\begin{lemma}\label{Se_degree}
Let $G$ be a \skg\ generated by $P$ and let $\epsilon > 0$ be fixed, and assume $W$ is connected and nonbipartite.  For sufficiently
large $t$ there is a constant $a>0$, depending only on $P$ and
$\epsilon$, such that for all $v \in \mathcal{S}_{\epsilon}$, $\prob{Z^{(v)} \in
  \mathcal{S}_{\epsilon}} \geq a$. 
\end{lemma}

To prove this Lemma, we make use of the following standard observation about binomial random variables.

\begin{obs}\label{O:sqrt}
Let $\alpha_1 > \alpha_2$ be fixed constants and let $p \in (0,1)$.  There
exists constants $c$ and $n_0$, depending on $\alpha_1,\alpha_2,$ and $p$ such
that if $n > n_0$, then 
 \[ \prob{\bin{n,p} \in [np - \alpha_1\sqrt{np}, np - \alpha_2\sqrt{np}]} > c.\]
\end{obs}

\begin{proof}[Proof of Lemma \ref{Se_degree}]
Let $v$ be an arbitrary vertex in $\mathcal{S}_{\epsilon}$. Consider a collection of independent,
identically distributed random variables, 
$X_1, \ldots, X_t$, taking on values in $\set{1,\ldots, k}$ each with
probability $p_i$, where $p_i \geq p > 0$ for all $i$.  Let
$Z_i$ be the count of the number of $i$'s in these variables, that is,
$Z_i = \sum_j \one_{X_j = i}$.  For $c>0$,
let $\mathcal{E}_i$ be the event that $p_it - 2c\sqrt{t} \leq Z_i \leq
p_it - c\sqrt{t}$.  We then have that, for all  
$j \neq i$, 
\begin{align*}
\expect{ Z_j \mid \mathcal{E}_i} &\geq \paren{t - \paren{p_it -
    c\sqrt{t}}} \frac{p_j}{1-p_i} \\
&= \paren{(1-p_i)t + c\sqrt{t}} \frac{p_j}{1-p_i} \\
&= p_jt + \frac{c p_j}{1-p_i}\sqrt{t} \\
&\geq p_jt + c p \sqrt{t}. 
\end{align*}

To apply this observation to the context of $Z^{(v)}$ we first
consider the unweighted graph $W'$ on $[k]$ where $i \sim j$ if and
only if there is an unweighted walk of length 2 between $i$ and $j$ in
$W$.  Since $W$ is connected and non-bipartite, $W'$ is connected and thus there
exists a breadth-first traversal of $W'$.  As noted above, by the
definition of $\mathcal{S}_{\epsilon}$, for every $i$ we have $(\sigma M)_i \geq
(1-\epsilon)\pi_i$.  Further, by the pigeonhole
principle, there is some index $i$ such that $(\sigma M)_i \geq \pi_i(1-\epsilon)
+ \frac{\epsilon}{k}$.  Let $s_1$ be one such index and let $s_1,
\ldots, s_k$ be a breadth-first traversal of $W'$ starting at $s_1$.

Recall that we may analyze $Z^{(v)}$ from the point of view of the
product distribution $\paren{X^{(1)}}^{\sigma_1 t} \times \cdots \times
\paren{X^{(k)}}^{\sigma_k t}$ where each $X^{(i)}$ is an independent random
variable that takes values in the set of neighbors of $i$ in $W$.
Let the random variables $Z_{ij}$ be the number of times that
$X^{(i)}$ takes on the value $j$.  We note that we can ignore the
indices that $X^{(i)}$ can not take on, and so define $p_i = \min_{j,
  p_{ij} \neq 0} \frac{p_{ij}}{c_i}$.  We recursively define the events
$\mathcal{A}_1,\ldots, \mathcal{A}_k$ as follows.  The event
$\mathcal{A}_1$ is the event that for all $u \sim_W s_1$,
$\expect{Z_{us_1}} - 2\alpha_1\sqrt{t} \leq Z_{us_1} \leq \expect{Z_{us_1}}
  - \alpha_1\sqrt{t}.$  For all $1 < i \leq k$ the event $\mathcal{A}_i$ is
  the event that for all $u\sim_W s_i$, 
\begin{equation}\label{abovecalculations} \expect{Z_{us_i} \mid \cap_{j=1}^{i-1} \mathcal{A}_j} -
2\alpha_i\sqrt{t} \leq Z_{us_i} \leq \expect{Z_{us_i} \mid \cap_{j=1}^{i-1} \mathcal{A}_j}
  - \alpha_i\sqrt{t}, \end{equation} where the $\alpha_i$'s are fixed constants to be chosen later.
We note that by Observation \ref{O:sqrt}, as $Z_{us_i}$ is a sum of independent indicator variables, each with probability $p_i$,  each of these events
occurs with positive probability. Thus it suffices to show that
$\cap_{i=1}^k \mathcal{A}_i$ is contained in the event $Z^{(v)} \in
\mathcal{S}_{\epsilon}$.  

For sufficiently large $t$ the event $\mathcal{A}_1$ assures that
$Z^{(v)}_{s_1} \geq (1-\epsilon)\pi_{s_1}$ by the choice of
$s_1$, specifically that $\expect{Z^{(v)}_{s_1}} \geq (1-\eps)\pi_{s_1} +
\frac{\epsilon}{k}$.

Since the sequence $s_i$ is a breadth-first search of $W'$, we have
that for all $i > 1$, there exists index $j <i$ such that $s_i
\sim_{W'} s_j$.  Thus there is some vertex $u$ that is a neighbor to
both $s_i$ and $s_j$ in $W$.    Now consider the effect of the
conditioning on the event $\mathcal{A}_j$ on $Z_{us_i}$. By (\ref{abovecalculations}) and the definition of $\mathcal{A}_j$ we have that 
$\expect{Z_{us_i} \mid \cap_{j=1}^{i-1} \mathcal{A}_j} \geq
\expect{Z_{us_i}} + \alpha_{i-1}p_{u}\sqrt{t} \geq \expect{Z_{us_i}} +
\alpha_{i-1}p_{\min}\sqrt{t}$ where $p_{\min} = \min_{i \in [k]} p_i$.
Furthermore, this gives that 
$t\expect{Z^{(v)}_{s_i} \mid \cap_{j=1}^{i-1} \mathcal{A}_j} \geq
  (1-\epsilon)\pi_{s_i}t + \alpha_{i-1}p_{\min}\sqrt{t}$.   Thus choosing
  $\alpha_i = \paren{\frac{2k}{p_{\min}}}^{k-i}$ suffices to assure that
  the event $\cap_{i=1}^k \mathcal{A}_i$ is contained in
  $\mathcal{S}_{\epsilon}$, as desired.
\end{proof}

\begin{theorem}\label{T:Sepsilon}
Let $G$ be a \skg\ generated by a matrix $P \in [0,1]^{k \times k}$
such that $W$ is connected and non-bipartite.  Further suppose that
$\sum_i c_i\ln(c_i) > 0$ and fix 
\[0 < \epsilon<\frac{\sum c_i\ln(c_i)}{\sum c_i\ln(c_i)-\vol{W}\ln(c_1)},\]
 if $\sum c_i\ln c_i\neq \vol{W}\ln(c_1)$, and $\eps>0$ if $\sum c_i\ln c_i= \vol{W}\ln(c_1)$.
 
Let $H$ be the subgraph of $G$ induced by $\mathcal{S}_{\epsilon}$.  For $t$
sufficiently large, there is a constant $d > 1$, depending on $P$
and $\epsilon$, such that $H$ is connected with diameter
$\bigOh{\ln\abs{\mathcal{S}_{\epsilon}}}$ with probability at least $1 - e^{-\bigTheta{d^t}}$.
\end{theorem}

Notice that the bound on $\eps$ is always positive (or infinite),
since $c_1\leq c_i$ for all $i$, so $\vol{W}\ln c_1=\sum c_i\ln(c_1) \leq \sum c_i\ln(c_i)$.

\begin{proof}  
We will proceed by showing that there exists a constant $c>0$ such that the graph $H$ has $\lambda_1(H)>c$ asymptotically almost surely. As noted in Theorem \ref{T:standards}, this implies that $H$ is connected asymptotically almost surely, with diameter $\bigOh{\ln(\size{\mathcal{S}_{\epsilon}})}$.

Recall that the expected degree of a vertex with signature $\sigma$ is
$\paren{c_1^{\sigma_1} \cdots c_k^{\sigma_k}}^t$ and thus any vertex
$v \in \mathcal{S}_{\epsilon}$ has expected degree at least \[c_1^{\epsilon
  t}\paren{c_1^{\pi_1} \cdots c_k^{\pi_k}}^{(1-\epsilon)t}
= \paren{c_1^{\epsilon} \paren{c_1^{c_1} \cdots
    c_k^{c_k}}^{\frac{1-\epsilon}{\vol{W}}}}^t = d^t,\] where \[d
= c_1^{\epsilon}\paren{c_1^{c_1}\cdots c_k^{c_k}}^{\frac{1-\epsilon}{\vol{W}}}.\]  
We note that by the restriction on $\epsilon$,
\begin{align*}
\ln(d) &= \epsilon \ln(c_1) + \frac{1-\epsilon}{\vol{W}}\sum_i
c_i \ln(c_i) \\
&= \frac{1}{\vol{W}}\sum_i c_i\ln(c_i) + \epsilon \paren{ \ln(c_1) -
  \frac{1}{\vol{W}}\sum_i c_i \ln(c_i)} \\
&> 0,
\end{align*}
and thus $d > 1$.  This implies that every vertex in
$\mathcal{S}_{\epsilon}$ has expected degree exponentially increasing with $t$.  

Let $\overline{H}$ be the subgraph of $W^{\otimes t}$ induced by
$\mathcal{S}_\epsilon$, so the weight of each edge in $\overline{H}$ is the probability of that edge appearing in
$H$. Now, by Lemma \ref{Se_degree}, there is some constant $c$ such
that for every vertex $v$ in $\overline{H}$ we have
$\deg_{\overline{H}}(v) \geq c d^t$. Now for any positive constant
$\delta$, there exists some small positive constant $c'$ such that  
\[ \frac{27 \ln\paren{\frac{4\size{\mathcal{S}_{\epsilon}}}{e^{-c' d^t}}}}{c
  d^t} \leq \frac{27 \ln\paren{\frac{4k^t}{e^{-c' d^t}}}}{c
  d^t} = \frac{27\paren{t\ln(k) + \ln(4) + c'd^t}}{cd^t} = \lilOh{1} +
\frac{27 c'}{c} \leq \delta^2,\] and thus, by Theorem
\ref{T:IndepSpec}, in order 
to complete the proof it suffices to show that $\overline{H}$ has
constant spectral gap. Indeed, by Theorem 3, if there exists a constant $\zeta$ with $\lambda_1(\overline{H})>\zeta>0$, then by Theorem 3, $\lambda_1(H)\geq \lambda_1(\overline{H})-\frac{27 \ln\paren{\frac{4\size{\mathcal{S}_{\epsilon}}}{e^{-c' d^t}}}}{c
  d^t}$ with probability at least $1-e^{-c'd^t}$, and by the above, we have that $\lambda_1(H)>\zeta-\delta^2$ for any $\delta>0$ with probability at least $1-e^{-c'd^t}$, as desired.

To determine the spectral gap in $\overline{H}$, we use Cheeger's
inequality. Let $X\subset \mathcal{S}_\epsilon$ with
$\Vol_{\overline{H}}(X)<\frac{1}{2}\Vol_{\overline{H}}(\mathcal{S}_\epsilon)$. Note
that
\[ h_{\overline{H}}(X)=\frac{e_{\overline{H}}(X,
  \mathcal{S}_\epsilon\ba X)}{\Vol_{\overline{H}}(X)} \geq
\frac{ce_{W^{\otimes t}}(X, V\ba X)}{\Vol_{W^{\otimes
      t}}(X)}= c\, h_{W^{\otimes t}}(X),\]
where the constant $c$ is the constant provided by Lemma \ref{Se_degree}.
Thus, we have
\begin{align*}
h_{\overline{H}} &= \min_{\substack{X\subset \mathcal{S}_\epsilon\\ \vol{X}<\frac{1}{2}\vol{\mathcal{S}_\epsilon}}} h_{\overline{H}}(X)\\
  &\geq c\min_{\substack{X\subset \mathcal{S}_\epsilon\\ \vol{X}<\frac{1}{2}\vol{\mathcal{S}_\epsilon}}} h_{W^{\otimes t}}(X)\\
  &\geq c\, h_{W^{\otimes t}}.\end{align*}

Now, let $M_1 = C^{-\half}PC^{-\half}$ and let $1=\mu_0\geq\mu_1\geq\cdots\geq \mu_{k-1}$ be the
eigenvalues of $M_1$. 
Note that $I-M_1$ is the
Laplacian matrix for $W$, and as $W$ is connected and non-bipartite,
$-1 < \mu_{k-1} \leq \mu_1 < 1$.
  Now, $\L\!\paren{W^{\otimes t}} = I-M_1^{\otimes t}$, and thus has eigenvalues
$1-\mu_{a_1}\mu_{a_2}\cdots\mu_{a_t}$, where $a_1, a_2,
\dots ,a_t\in [k-1]\cup\{0\}$.  Hence, the smallest nonzero
eigenvalue of $\L\paren{W^{\otimes t}}$ is $1-\mu_1$, which occurs with
multiplicity $t$. Thus by Cheeger's inequality, $h_{W^{\otimes t}}\geq
\frac{1-\mu_1}{2}$.

Therefore, combining these results we have
\[\lambda_1\!\paren{\overline{H}}\geq \frac{1}{2} h_{\overline{H}}^2
\geq \frac{c}{2} h_{W^{\otimes t}}^2 \geq
\frac{c^2}{8}(1-\mu_1)^2.\]
Hence $\lambda_1\!\paren{\overline{H}}$ is bounded below by a
constant and $\overline{H}$ has constant spectral gap, as desired.
\end{proof}

This establishes that the graph $G$ contains a small connected core
asymptotically almost surely provided $\sum c_i\ln c_i>0$. We now turn our attention to the second half of our fundamental structure. Here we wish
to determine which vertices will be connected by a path to the
connected core. To that end, define $\Sigma_{\nu}= \set{v \in V(G)
  \mid \inner{\sigma(v)M^s}{L} \geq \nu \textrm{ for all } s \geq 0 }$.
 We wish to show that any vertex in $\Sigma_{\nu}$ may be connected by
 a path to $\mathcal{S}_\eps$ asymptotically almost surely.

\begin{theorem}\label{T:SnakeOne}
Let $G$ be a \skg\ generated by a matrix $P \in [0,1]^{k \times k}$
such that $W$ is connected and non-bipartite.  Fix 
$0 < \epsilon, \nu$.  Let $\lambda$ be the spectral gap of $W$ and let
$s = \ceil{\frac{1}{\lambda}\ln\paren{\frac{2\vol{W}}{c_1\epsilon}}}$.
For $t$ sufficiently large, any vertex $v \in \Sigma_{\nu}$ is
connected to $\mathcal{S}_{\epsilon}$ by a path of length at most $s$ with
probability at least $1-se^{-e^{\nu t - \bigTheta{\sqrt{t}}}}$.  
 \end{theorem}

\begin{proof}
Let $v \in \Sigma_{\nu}$.  Define $v_0 = v$ and for each $1 \leq i
\leq s$, let $v_i$ be a neighbor of $v_{i-1}$ such that
$\norm[]{\sigma(v_i) - \sigma(v_{i-1})M} \leq \sqrt{\frac{\ln(6k)}{2t}}$ (if such a neighbor exists). For $1 \leq i \leq
s$ define $\eta_i = \sigma(v_i) - \sigma(v_{i-1}M)$.  Now, we note
that if such a sequence exists, then 
\[ 
\norm[]{\sigma(v_j) - \sigma(v)M^j} \leq \norm[]{\sum_{i=
    1}^j \eta_iM^{j-i}} 
\leq \sum_{i=1}^j \norm{\eta_iM^{j-i}} 
\leq \sum_{i=1}^j \norm[1]{\eta_i} 
\leq \sum_{i=1}^j k \sqrt{\frac{\ln(6k)}{2t}} 
= jk \sqrt{\frac{\ln(6k)}{2t}}, \] 

and further \[\inner{v_j}{L} \geq \inner{v_0}{L} -
jk \sqrt{\frac{\ln(6k)}{2t}}\norm[1]{L} \geq \nu - jk
\sqrt{\frac{\ln(6k)}{2t}}\norm[1]{L}.\]
Thus, since $s$ is a fixed constant, we have that by Corollary
\ref{C:tight} for sufficiently large $t$ such a sequence fails to
exist with probability at most 
 \[s e^{ - \frac{ \paren{e^{\nu -
        sk\sqrt{\frac{\ln(6k)}{2t}}\norm[1]{L}}}^t}{12}} = se^{- \frac{e^{\nu t
      - \bigTheta{\sqrt{t}}}}{12}} =  se^{-e^{\nu t - \bigTheta{\sqrt{t}}}}\]
It now suffices to show that $v_s \in \mathcal{S}_{\epsilon}$. 

By the choice of $s$ and Theorem \ref{TVDist}, we know that \[
\abs{\frac{\paren{\sigma(v) M^s}_i - \pi_i}{\pi_i}} \leq \frac{\epsilon}{2},\]
and thus $\paren{\sigma(v)M^s}_i \geq (1-\frac{\epsilon}{2})\pi_i.$  But then
as $\abs{(v_s)_i - \paren{\sigma(v)M^s}_i} \leq sk\sqrt{\frac{\ln(6k)}{2t}}$ we
have that for sufficiently large $t$, $v_s \in \mathcal{S}_{\epsilon}$.     
\end{proof}

\section{Small Components}

We now turn to the case that the stochastic Kronecker graph has only small components, that is, the largest component is of size at most $o(n)=o(k^t)$. These correspond to items (\ref{M:disc_bip}) and
(\ref{M:small}) in Theorem \ref{T:Master}.  The first of these result
follows from standard results on the component sizes of
(non-stochastic) Kronecker graphs which we include in the following
lemma for completeness.

\begin{lemma}\label{T:smallcomps}
If $H$ is a disconnected or bipartite graph on $k$ vertices, then the largest
component of $H^{\otimes t}$ has size $\bigOh{(k-1)^t}$.  
\end{lemma}
\begin{proof}
First, suppose $H$ is not connected.  Let $v=(v_1, v_2, \dots, v_t)$
be a vertex in $H^{\otimes t}$.  Now for any neighbor $u = (u_1,u_2,
\ldots, u_t)$ of $v$ each coordinate $u_i$ must be adjacent to $v_i$
in $H$ and hence in the same component as $v_i$.  Thus, the size of
the component containing $v$ is at most the product of the sizes of
the components in $H$ of the vertices $v_i$.  Since $H$ is
disconnected the largest component in $H$ has size at most $k-1$ and
thus the largest component in $H^{\otimes t}$ has size at most $(k-1)^t$.

Now, suppose $H$ is a connected bipartite graph with bipartition $(A,B)$ and again
consider a vertex $v = (v_1,v_2,\ldots,v_t)$ and a neighbor $u$ of $v$, with $u
= (u_1,u_2,\ldots, u_t)$. Now since $v_i$ and $u_i$ are adjacent in
$H$, they are on different sides of the bipartition $(A,B)$.  Thus the
component containing $v$ and $u$ is bipartite with $u$ and $v$ on
different sides of the bipartition.  Furthermore, the side of the
bipartition containing $v$ has $\size{A}^{\size{\set{i : v_i \in
    A}}}\size{B}^{\size{\set{j : v_j \in B}}}$ vertices.  Thus for all $0 \leq i \leq
t$ there are $\binom{t}{i}$ components of $H^{\otimes t}$ of size  $\size{A}^i\size{B}^{t-i} +
\size{A}^{t-i}\size{B}^i$.  It is worth noting that this size is
symmetric and so that components counted for a given $i$ are also
counted for $t - i$.  Now maximizing $\size{A}^i\size{B}^{t-i} +
\size{A}^{t-i}\size{B}^i$ over the choice of $i$, we have the largest
component occurs where either $i=0$ or $i=t$. As $\size{B}=k-\size{A}$, we maximize with respect to $\size{A}$ to obtain that the largest of component of $H^{\otimes t}$ has size at most
$(k-1)^t+1$ for $k>1$. 
\end{proof}

This lemma resolves item (\ref{M:disc_bip}) in Theorem
\ref{T:Master} as it implies that the underlying graph for
$P^{\otimes t}$ is disconnected with small component sizes. 

\begin{theorem}\label{T:small}
Let $G$ be a \skg\ generated by $P \in [0,1]^{k \times k}$ with column sums $c_1 \leq
\cdots \leq c_k$.  If $W$ is connected, non-bipartite, and $\prod_i c_i <
1$,
then there exists some $0 < \delta < 1$ such that 
with probability at least $1 - e^{-\frac{n^{\delta}}{3}}$ there are at least $n -
\bigOh{n^{\delta}}$ isolated vertices in $G$.
\end{theorem}

\begin{proof}We consider two cases, according to whether $c_k=c_1$ or $c_k\neq c_1$.

If $c_k\neq c_1$, then as $\prod_i c_i < 1$ we may set $0<\eps =\frac{-1}{k}\sum_i \ln(c_i)$, so that $\sum_i \ln(c_i) = -\epsilon k <
0$. Let $\alpha$ be a solution to \[ \alpha = \frac{2(\epsilon -
  \alpha)^2}{\paren{\ln(c_k) - \ln(c_1)}^2} \] in the interval
$[0,\epsilon]$.  Such an $\alpha$ exists as $\alpha$ and $ \frac{2(\epsilon -
  \alpha)^2}{\paren{\ln(c_k) - \ln(c_1)}^2}$ are continuous functions,
$0 < \frac{2\epsilon^2}{\paren{\ln(c_k) - \ln(c_1)}^2}$, and $\epsilon
> 0$.  Let $\delta = 1 - \frac{\alpha}{\ln(k)}$.  Let $X= X_1 + \cdots
+ X_t$
where each $X_i$ takes values independently uniformly from $\set{\ln(c_1), \ldots,
  \ln(c_k)}$.  Note that $X$ can be thought of as the natural
logarithm of the expected degree of a vertex of $G$ chosen uniformly
at random.  Now by Hoeffding bounds we have that \[\prob{X \geq -\alpha
  t}  = \prob{X + \epsilon t \geq (\epsilon - \alpha)t} \leq
e^{-\frac{2(\epsilon - \alpha)^2}{\paren{\ln(c_k) - \ln(c_1)}^2}t} =
e^{-\alpha t}.\]  Thus there are at most $k^t e^{-\alpha t} = n^{\delta}$ vertices
of $G$ with expected degree larger than $e^{-\alpha t}$.  The sum of
the expected degrees of vertices with expected degree smaller than
$e^{-\alpha t}$ is at most $k^te^{-\alpha t} = n^{\delta}$.  Thus by
Chernoff bounds with probability at least
$1-e^{-\frac{n^{\delta}}{3}}$ there are at most $2n^{\delta}$ edges
incident to vertices with expected degree at most $e^{-\alpha t}$.
Combining this with the vertices with expected degree at least
$e^{-\alpha t}$ we have that there are at most $3n^{\delta}$
non-isolated vertices in $G$.      

\quad

For the second case, if $c_k=c_1$, then we note that $c_1=c_2=\dots=c_k$. As $\prod c_i<1$, we have that $c_1<1$, and the expected degree of every vertex in $G$ is $c_1^t$. Note then that by linearity of expectation, we have that the expected number of edges in $G$ is $\frac{1}{2}n c_1^t = \frac{1}{2}n^{1+\frac{\ln c_1}{\ln k}}$. As $\ln c_1<0$, we have that the expected number of edges in $G$ is at most $n^\delta$ for some $0<\delta<1$. By Chernoff bounds, then, the number of edges in $G$ is at most $\frac{3}{2}n^\delta$ with probability at least $2\exp\paren{-n^\delta/8}=\lilOh{1}$. But then the number of nonisolated vertices in $G$ is at most $3n^\delta$, and the result follows.

\end{proof}

The preceding theorem resolves item (\ref{M:small}) in Theorem \ref{T:Master}. 

\section{Giant Components}

We now turn our attention to proving item (\ref{M:giant}) in Theorem \ref{T:Master}. To prove this result, we will use the structure outlined in Section \ref{Structural}, and in particular, Theorems \ref{T:Sepsilon} and \ref{T:SnakeOne} regarding the existence of a connected core of vertices and the vertices that can be connected by a path to $S_\eps$. In order to apply these theorems, however, we must verify that the conditions are met. We thus begin with several additional lemmas addressing the case that $\prod_i c_i>1$.

\begin{lemma}\label{L:c_prod}
Let $0 < c_1 \leq \cdots \leq c_k$ be such that $\prod_i c_i \geq 1$.
Then $\sum_i c_i \ln(c_i) \geq 0$ with equality if and only if the
$c_i$'s are identically 1. 
\end{lemma}

\begin{proof}
Define $\delta_j = c_j - c_{j-1} \geq 0$, where $c_0$ is defined to be 0 and
define $s_j = \sum_{i=j}^k \ln(c_i)$. As $\sum_i c_i \ln(c_i) =
\sum_i \delta_i s_i$, and all the $\delta_i \geq 0$, it suffices
to show that $s_i \geq 0$ for $i=1, 2, \dots, k$.  We note that since the $c_i$'s are
increasing and $\ln(\cdot)$ is a monotonically increasing function
$0 \leq \sum_i \ln(c_i) \leq \frac{j-1}{k-j+1}s_j + s_j$,
and thus $s_j \geq 0$ for all $j$. 

 We note that if $\prod_i c_i
> 1$, then the previous argument implies that $\sum_i c_i\ln(c_i) > 0$.  
Thus suppose that $\prod_i c_i = 1$ and yet the $c_i$'s are not
identically 1.  As this implies that $c_k > 1$ and $c_1 < 1$, there is
some minimal $j$ such that $c_j > 1$. But then as $c_{j-1} \leq 1$,
$\delta_j > 0$ and $s_j = \sum_{i=j}^k \ln(c_i) \geq (k-j+1) \ln(c_j)
> 0$, we have that $\sum_i c_i \ln(c_i) > 0$, as desired.    
\end{proof}

\begin{lemma}\label{L:uniform}
Let $P$ be a symmetric matrix in $[0,1]^{k \times k}$ with
non-identical column sums
$0 < c_1 \leq \cdots \leq c_k$.  Further suppose that the
associated weighted graph $W$ is connected and non-bipartite.  Let  $f$ be a
strictly monotonically increasing function on $\R^+$ and let $L$ be
the vector $\paren{f(c_1), \ldots, f(c_k)}$.  If $M$ is the transition
matrix for the uniform random walk on $W$, then $\inner{\one M^s}{L} >
\inner{\one}{L}$ for all $s \geq 1$. 
\end{lemma}
\begin{proof}
We first note that  $M = C^{-1}P$ and consider
\begin{align*}
\inner{\one M}{L} - \inner{\one}{L} &= \inner{\one C^{-1}P}{L} -
\inner{\one}{L} \\ &= \sum_{i=1}^k \sum_{j=1}^k  \frac{P_{ij}}{c_i}L_j - \sum_{j=1}^k L_j \\
&= \sum_{i=1}^k \sum_{j=1}^k  \frac{P_{ij}}{c_i}L_j - \sum_{j=1}^k
\sum_{i=1}^k 
\frac{P_{ij}}{c_j}L_j \\
&= \sum_{i=1}^k \sum_{j=1}^k  \paren{\frac{P_{ij}}{c_i} -
  \frac{P_{ij}}{c_j}} L_j \\
  &= \sum_{c_i > c_j} P_{ij}\paren{\frac{1}{c_j}-\frac{1}{c_i}}\paren{L_i-L_j}\\
\end{align*}
Note that as $f$ is strictly increasing, $L_i - L_j > 0$ and
$\frac{1}{c_j} - \frac{1}{c_i} > 0$ for $c_i > c_j$.  Further, as $W$ is connected, $P_{ij} > 0$ for some $i$ and $j$ with $c_i \neq
c_j$, giving that $\inner{\one M}{L} - \inner{\one}{L} > 0$.  

To complete the proof it would suffice to show that $M^s$ is the
transition probability matrix for the uniform random walk on some
connected, non-bipartite graph with the same degree sequence as $W$.
To that end, fix some $s \geq 2$ and note that $M^s = C^{-1}\paren{PC^{-1}}^{s-1}P$, and
so let $P' = \paren{PC^{-1}}^{s-1}P$.  It is clear that $P'$ is
symmetric and has the desired column sums, thus it suffices to show
that the associated graph $W'$ is connected and non-bipartite.  We
note that $P'_{ij} > 0$ if and only if there is a length $s$ walk
between $i$ and $j$ in $W$.  We note that if $s$ is odd, then the
edges present in $W'$ are a superset of the edges in $W$, and thus
$W'$ is connected and non-bipartite.  

Thus suppose $s$ is even and let
$\mathcal C$ be an odd length cycle in $W$.  Consider the walk in $W'$ formed
by starting at vertex $v$ and traversing the cycle $\mathcal C$ in steps of length $s$.  As $s$ is
even and the length of the cycle is odd, it will take an odd number of steps
in $W'$ to return to the vertex $v$.  Thus, there is a closed walk in
$W'$ of odd length and hence $W'$ is non-bipartite.   We note that
as $s$ is even $W'$ contains self-loops at all vertices and edges
between pairs of vertices that are connected by a walk of length 2.
Thus in order to show that $W'$ is connected it suffices to show that
there is an even length walk between any two vertices in $W$.
For any two distinct vertices $u$ and $v$ in $W$ such a walk can be
constructed by taking a walk from each vertex to the odd cycle $\mathcal C$ and
then traversing $\mathcal C$ in both directions.  As $\mathcal C$ is an odd cycle, these
two traversals will have opposite parity, and thus one of those walks
will have even length.  
\end{proof}

These two Lemmas immediately give part (\ref{M:giant}) of our main theorem, as follows.
\begin{theorem}\label{T:giant}
Let $G$ be a \skg\ generated by a matrix $P \in [0,1]^{k \times k}$
such that $W$ is connected and non-bipartite. If $\prod_i c_i >
1$, then there are constants $s, d > 1$, depending only on $P$, such that
for sufficiently large $t$, $G$ has a giant component with probability
at least $1 - s k^t e^{-\bigTheta{d^t}}$.  
\end{theorem}

\begin{proof}
First, if $c_1=c_2=\dots=c_k$, we note that the minimum degree in $W^{\otimes t}$ is at least exponential in $t$, and hence by Theorem \ref{T:IndepSpec} together with the spectral properties of Kronecker products used in Theorem \ref{T:Sepsilon}, $G$ is connected with probability at least $1-e^{-\bigTheta{d^t}}$, and the result follows immediately.

If not, then by Lemma \ref{L:c_prod}, we have that $\sum_i c_i\ln(c_i) > 0$. Fix 
\[ 0 < \epsilon = \frac{\sum_i c_i \ln(c_i)}{2\sum_i c_i
  \ln(c_i) - 2\ln(c_1)\vol{W}} <  \frac{\sum_i c_i \ln(c_i)}{\sum_i c_i
  \ln(c_i) - \ln(c_1)\vol{W}}.\]
By Theorem \ref{T:Sepsilon}, there is some constant $d_1 > 1$ which
depends only on $P$ such that $\mathcal{S}_{\epsilon}$ is connected with
probability at least $1 - e^{-\bigTheta{d_1^t}}$.

Fix some positive constant $c$.  Let $v$ be an arbitrary vertex such
that $\norm[]{\sigma(v) - \frac{1}{k}\one} \leq
\frac{c}{\sqrt{t}}$ and let $\eta_v = \sigma(v) - \frac{1}{k}\one$.
Noting that $\inner{\one}{L} = \ln\paren{\prod_i c_i} > \ln(1) =
0$, we have that for sufficiently large $t$ and all $s \geq 0$, 

\begin{align*}
\inner{\sigma(v)M^s}{L} &= \inner{\paren{\frac{1}{k}\one + \eta_v} M^s}{L} \\
&= \frac{1}{k}\inner{\one M^s}{L} + \inner{\eta_v M^s}{L} \\
&\geq \frac{1}{k}\inner{\one}{L} - \norm[1]{\eta_v}\norm[\infty]{L} \\
&\geq \frac{1}{k}\inner{\one}{L} - \frac{k c
  \norm[\infty]{L}}{\sqrt{t}} \\
&> \frac{1}{2k}\inner{\one}{L},
\end{align*}
where the first inequality follows from Lemma \ref{L:uniform}.  Let
$d_2 = e^{\frac{1}{2k}\inner{\one}{L}}$ and note that this implies that 
$v \in \Sigma_{\frac{1}{2k}\inner{\one}{L}}$ and so by Theorem
\ref{T:SnakeOne} there is a constant $s$ such that  with probability
at least $1 - s e^{-(\frac{1}{12}-\lilOh{1}) d_2^t}$ the vertex $v$ is
connected to $\mathcal{S}_{\epsilon}$ by a path of length at most $s$.
Observing that a constant fraction of the vertices have the desired
signature by Chernoff bounds completes the proof.
\end{proof}

A slight modification of this argument gives part
(\ref{M:giant_bound}) of the main theorem.
\begin{theorem}\label{T:giant_bound}
Let $G$ be a \skg\ generated by a matrix $P \in [0,1]^{k \times k}$
such that $W$ is connected and non-bipartite. If $\prod_i c_i =
1$ such that the $c_i$'s are not all equal, then there are constants $s, d > 1$, depending only on $P$, such that
for sufficiently large $t$, $G$ has a giant component with probability
at least $1 -  e^{-\bigTheta{d^t}}$.  
\end{theorem}

\begin{proof}
Since the $c_i$'s are not all equal, we have that $\sum_i c_i\ln(c_i) >
0$ by Lemma \ref{L:c_prod}.  Fix $\eps>0$ satisfying the hypotheses of Theorem \ref{T:Sepsilon}; we shall further restrict $\eps$ as needed below. Then by Theorem \ref{T:Sepsilon} there is some constant $d_1
> 1$ such that $\mathcal{S}_{\epsilon}$ is connected with probability
at least $1 - e^{-\bigTheta{d_1^t}}$.  

Let $s = \ceil{\frac{1}{\lambda}\ln\paren{\frac{2\vol{W}}{\epsilon
      c_1}}}$ and note that by Theorem \ref{TVDist}, $\frac{1}{k}\one
  M^{j} \in \mathcal{S}_{\nicefrac{\epsilon}{2}}$ for all $j \geq s$.

Note that for all $i=1, 2, \dots, k$, if $\sigma\in \mathcal{S}_{\nicefrac{\eps}{2}}$, then we have \[\sigma_i=1-\sum_{j\neq i}\sigma_j\leq 1-\sum_{j\neq i}\left(1-\frac{\eps}{2}\right)\pi_j = 1-\left(1-\frac{\eps}{2}\right)(1-\pi_i) = \left(1-\frac{\eps}{2}\right)\pi_i+\frac{\eps}{2}.\]
Let $l$ be the largest integer such that $c_l<1$. 
  Then we have, for all $j\geq s$, \begin{eqnarray}
\nonumber  \inner{\frac{1}{k}\one M^{j}}{L} & \geq& \sum_{i=1}^l  \left[\left(1-\frac{\eps}{2}\right)\pi_i+\frac{\eps}{2}\right]\ln c_i + \sum_{i=l+1}^k  \left(1-\frac{\eps}{2}\right)\pi_i\ln c_i\\
\nonumber &=&  \sum_i \left(1-\frac{\eps}{2}\right)\pi_i\ln c_i+\frac{\eps}{2}\sum_{i=1}^l \pi_i\ln c_i\\
\label{epsilonthing}& \geq & \sum_i\left(1-\frac{\eps}{2}\right)\pi_i\ln c_i + \frac{\eps}{2}k\ln c_1
  \end{eqnarray} for all $j \geq s$. Here we further restrict $\eps$ to be sufficiently small that the quantity in (\ref{epsilonthing}) is positive. Moreover, if $1\leq j<s$, we may apply Lemma \ref{L:uniform} to obtain a constant $\mu>0$ such that $\inner{\frac{1}{k}\one M^j}{L}\geq \mu$ for all $1\leq j<s$. Since $s$ is a fixed
  constant, this implies that there is some $\nu > 0$ such that for
  all $j \geq 1$, we have $\inner{\frac{1}{k}\one M^j}{L} \geq \nu$.

Let $c$ be a constant to be fixed later.  We notice that for $t$ sufficiently
large  all vertices $v$ such that
$\norm{\sigma(v) - \frac{1}{k}\one M} \leq \frac{c}{\sqrt{t}}$  are contained
in $\Sigma_{\nicefrac{\nu}{2}}$.  Thus by Theorem \ref{T:SnakeOne}
these vertices
are connected to $\mathcal{S}_{\nicefrac{\epsilon}{2}}$ with
probability at least $1-e^{-\bigTheta{d_2^t}}$ where $d_2 =
  e^{\nicefrac{\nu}{2}}$. 

 At this point it suffices to show that a
  constant fraction of the vertices in $G$ are adjacent to
  $\Sigma_{\nicefrac{\nu}{2}}$.  To this end, let $V'$ be the set of vertices $v $
  such that $\abs{\sigma_j(v) - \frac{1}{k}} \leq
\frac{1}{k\sqrt{t}}$ for $1 \leq j < k$ and $\abs{\sigma_k(v) -
  \frac{1}{k}} \leq 
\frac{-\ln(c_1)}{\ln(c_k)\sqrt{t}}$.  By Chernoff bounds and
Observation \ref{O:sqrt}, we have that a constant fraction of the
vertices of $G$ are in $V'$.  Furthermore, for every vertex $v \in
V'$, $\expect{\deg(v)} \geq 1$.   Now by part (\ref{concinGbar}) of
Lemma \ref{agoodneighbor}, for all $v \in V'$,

\[ \sum_{\norm{\sigma(u) - \sigma(v)M} \leq \sqrt{\frac{\ln(2k)}{2t}}}
\prob{u \sim v} \geq (1-e^{-1}) \expect{\deg(v)} \geq 1 - e^{-1}.\]

Thus, any fixed vertex in $v \in V'$ has a neighbor $u$ such that
$\norm{\sigma(u) - \sigma(v)M} \leq\sqrt{\frac{\ln(2k)}{2t}} $ with
probability at least $e^{-2(1-e^{-1})}$.  We further note that any such neighbor is a member of $\Sigma_{\nicefrac{\nu}{2}}$.  Taking $c \geq
\frac{1}{\sqrt{2}} + \max\set{\frac{1}{k},
  \frac{-\ln(c_1)}{\ln(c_k)}}$ and applying Chernoff bounds completes
the proof.
\end{proof}

\section{Connectivity}

Finally, we turn to the connectivity of $G$.  We note that part
(\ref{M:connected}) of the main theorem follows immediately from Theorem
\ref{T:IndepSpec} by observing that the minimum degree in $W^{\otimes
  t}$ is exponential in $t$ and exploiting the spectral properties of
the Kronecker product, as mentioned in the proof of Theorem \ref{T:giant}.  However, in keeping with the theme of this
paper we provide an alternative proof which exploits the Markov chain
structure.

\begin{theorem}\label{T:const_diam}
Let $G$ be a \skg\ generated by a matrix $P \in [0,1]^{k \times k}$
such that $W$ is connected and non-bipartite. If $1 < c_1 \leq \ldots
\leq c_k$, then there is some constant $d > 1$, depending only on $P$
such that $G$ is connected with probability at least $1 - e^{-\bigTheta{d^t}}$.
\end{theorem}

\begin{proof}
We first note that as $c_1 > 1$, $\ln(c_1) > 0$ and thus for any
signature $\sigma$, $\inner{\sigma}{L} \geq \ln(c_1) > 0$.  Thus every
vertex is in $\Sigma_{\ln(c_1)}$ and hence by Theorem \ref{T:SnakeOne}
for every $\epsilon > 0$, every vertex is connected to $\mathcal{S}_{\epsilon}$
by a path of constant length with probability at least $1 -
ne^{-c_1^{(1-\lilOh{1})t}}$.  Thus it suffices to show that there is
some $\epsilon > 0$ such that $\mathcal{S}_{\epsilon}$ is connected.  But as
$c_i > 1$ for all $i$, this implies that $\sum_i c_i\ln(c_i) > 0$ and
thus by Theorem \ref{T:Sepsilon} there is some constant $\hat{d} > 1$,
depending only on $P$, such that $\mathcal{S}_{\epsilon}$ is connected with
probability at least $1-e^{-\bigTheta{\hat{d}^t}}$.
   \end{proof}

The following two theorems address the case that $c_1=1$. We note that we will always have a giant component in this case, unless $c_1=c_2=\dots =c_k=1$. However, the connectivity no longer depends entirely on the degrees in the graph, but is determined based on how the weight is distributed among the vertices. In particular, the backbone graph will determine the behavior.

\begin{theorem}\label{T:discon}
Let $G$ be a \skg\ generated by $P \in [0,1]^{k \times k}$ with column
sums $1= c_1 \leq
\cdots \leq c_k$.  If $W$ is connected and non-bipartite and the
backbone graph $B$ has a vertex of degree zero, then there is a
constant $p \in (0,1)$ such that with probability at least $1 - p^t$
the graph $G$ has at least $\frac{1}{2}t^{(1-\lilOh{1})\ln\ln(t)}$
isolated vertices.
\end{theorem}

\begin{proof}
Note that as the backbone graph $B$ has a vertex of degree 0, there exists a vertex $v\in G$ such that for all vertices $u$, $\prob{u \sim v} \leq
\frac{1}{2}$. We note that in this case we have 
\begin{align*}
\ln\paren{\prob{\deg(v) = 0}} &= \ln\paren{\prod_{u} 1 - \prob{u \sim
    v}} \\
&= \sum_u \ln\paren{1 - \prob{u\sim v}} \\
&\geq - \sum_u \frac{\prob{u \sim v}}{1-\prob{u \sim v}} \\
&\geq - \sum_u 2\prob{u \sim v} \\
&= -2\expect{\deg{v}},
\end{align*}
where the last inequality comes from the upper bound on $\prob{u \sim
  v}$.  Thus we have that $\prob{\deg(v) = 0} \geq
e^{-2\expect{\deg(v)}}.$  Thus it suffices to find a large collection
of vertices in $G$ whose degrees are independent and where
$\expect{\deg(v)}$ is small.  

To that end suppose that there is some $i$ such that $p_{1i} = 1$,
that is, the degree of vertex $1$ in $B$ is not zero.  Thus there
is some $j \neq 1, i$ such that $j$ has degree zero in $B$.  Now let
$S^{(j)}_{t_j}$ be the set of vertices in $G$ whose signature $\sigma$ has $\sigma_j=\frac{t_j}{t}$, $\sigma_1=1-\frac{t_j}{t}$, and $\sigma_i=0$ for $i\neq 1, j$. Since $c_1 = 1$ and $p_{1i}=1$,
we know that $p_{1j} = 0$ and thus the degrees of all vertices in $S^{(j)}_{t_j}$ are
independent.  We note that there is a choice of constant $c$ such that
if $t_j = c\ln\ln(t)$ then the expected number of isolated vertices in
$S^{(j)}_{t_j}$ is $t^{(1-\lilOh{1})\ln\ln(t)}$, and thus by Chernoff bounds
with probability at least $1 -
e^{-\frac{t^{(1-\lilOh{1})\ln\ln(t)}}{6}}$ there are at least
$\frac{1}{2}t^{(1-\lilOh{1})\ln\ln(t)}$ isolated vertices in $G$.  

Now suppose that the degree of $1$ in $B$ is zero. Choose some index
$j \neq 1$ arbitrarily and consider the set $S^{(j)}_{t_j}$ as above.
As $j$ is arbitrary there may be some edges between vertices of
$S_{t_j}^{(j)}$.  Thus we note that when $3t_j \leq t$, we have 
\begin{align*}
\expect{e(S^{(j)}_{t_j},S^{(j)}_{t_j})} &= \sum_{u \in S^{(j)}_{t_j}}
\sum_{v \in S^{(j)}_{t_j}} \prob{u \sim v} \\
&= 2\binom{t}{t_j} \sum_{i=0}^{t_j} \binom{t_j}{i} \binom{t-t_j}{t_j-i}p_{jj}^{t_j - i} p_{j1}^{i}
p_{1j}^i p_{11}^{t- t_j - i} \\
&\leq 2\binom{t}{t_j}\binom{t-t_j}{t_j} p_{11}^{t- 2t_j}\paren{p_{11}p_{jj} + p_{1j}^2}^{t_j}\\
&\leq 2t^{t_j}t^{t_j}p_{11}^{t-2t_j}(p_{11}p_{jj}+p_{1j}^2)^{t_j}
\end{align*}

In particular, there is a constant $c'$ such that
$\expect{e(S^{(j)}_{t_j},S^{(j)}_{t_j})}
\leq \paren{c't^2}^{t_j}p_{11}^t$.  As $p_{11} < 1$, this implies that
the probability of an edge in $S_{t_j}^{(j)}$ is exponentially small
provided $t_j \in \lilOh{\frac{t}{\ln(t)}}$.  Thus, again choosing $t_j
= c \ln\ln(t)$ and conditioning on $e(S^{(j)}_{t_j},S^{(j)}_{t_j}) =
0$ gives the desired result. 
\end{proof}

A slight simplification of this result gives part (\ref{M:discon_c1})
of Theorem \ref{T:Master}.

\begin{theorem}\label{T:Bconn}
Let $G$ be a \skg\ generated by a matrix $P \in [0,1]^{k \times k}$
such that $W$ is connected and non-bipartite. If $1 = c_1 \leq \ldots
\leq c_k$ and the backbone graph $B$ has no vertices of degree zero,
then there is a constant $d> 1$ such that $G$ is connected with
probability at least $1-e^{-\bigTheta{d^t}}$.
\end{theorem}

\begin{proof}
First we note that $c_k > 1$ as otherwise the only edges present in $W$ are those present in the backbone graph, and in particular, $W$ is a perfect
matching contradicting the non-bipartiteness.  Thus we have that
$\sum_i c_i\ln(c_i) > 0$ and thus by Theorem \ref{T:Sepsilon} there is some
$\epsilon > 0$ and $d' > 1$ such that $S_{2\epsilon}$ is connected with probability at least $1 -
e^{-\bigTheta{d'^t}}$. 

 Now in a similar manner as the proof of Theorem \ref{T:SnakeOne} it
 suffices to show that asymptotically almost surely, from every vertex
 $v = v_0$ there is a sequence $v_0, v_1, \ldots, v_s$ such that
 $v_{i} \sim v_{i+1}$ and $v_s \in \mathcal{S}_{\epsilon} \subset
 S_{2\epsilon}$. By imposing the additional condition that  $\norm[\infty]{\sigma(v_i)M - \sigma(v_{i+1})} \leq \frac{\epsilon}{s
   k \norm[\infty]{L}}$, we may take $s= \ceil{\frac{1}{\lambda}\ln\paren{\frac{\vol{W}}{2\epsilon}}}$ by Theorem  \ref{TVDist} and the Markov chain viewpoint.

To that end fix an arbitrary vertex $v$ and consider the behavior of
$Z^{(v)}$ from the point of view of the product distribution 
$\paren{X^{(1)}}^{t_1} \times \cdots \times \paren{X^{(k)}}^{t_k}$
where $t_i$ is the number of $i$'s in the label for $v$.  
Notice that for
those indicies $i$ where $c_i = 1$, $X^{(i)}$ is the identity
distribution. Furthermore, these coordinates perfectly respect the
action of the Markov chain given by $M$. 

Let $j$ be the first index with $c_j>1$, so that $c_{j-1}\leq 1$. Suppose that $t_j +
\cdots + t_k \leq \frac{\epsilon}{s
   k \norm[\infty]{L}} t$. Note then any neighbor $u$ of $v$ in
$B^{\otimes{t}}$ immediately satisfies that $\norm[\infty]{\sigma(v)M - \sigma(u)}
\leq \frac{\epsilon}{s
   k \norm[\infty]{L}}$.

   Otherwise, we have $t_j+\cdots+t_k>\frac{\eps}{sk\norm[\infty]{L}}$. But then $\expect{\deg(v)} \geq c_j^{\frac{\epsilon}{s
   k \norm[\infty]{L}}t}$ and $c_j>1$, and thus by Lemma \ref{agoodneighbor}, there is a constant $c$ such
that 
\[ \sum_{\norm[\infty]{\sigma(v)M - \sigma(u)}\leq \frac{\eps}{sk\|L\|_\infty}} \prob{u \sim v} \geq c
c_j^{\frac{\epsilon}{s
   k \norm[\infty]{L}}t}.\]   
Applying Chernoff bounds to assure the existence of such a vertex
completes the proof.
\end{proof}

\section{Concluding Remarks}

We note that in principle these techniques can be extended to analyze
the emergence of connectivity and the giant component in
generalizations of the stochastic Kronecker graph, such as the
multiplicative attribute graph~\cite{Leskovec:MAG}.  In fact, based on
the work in \cite{Radcliffe:MAGspectra}, it is likely that similar
transition points will hold.  That is, the multiplicative attribute
graph will have a giant component when the median expected degree is 1
and become connected when the probability of an isolated vertex goes
to zero.  

Perhaps a more interesting direction would to resolve the size of the
largest component in the case when $c_1 = c_2 = \cdots = c_k = 1$.
By letting $P = \frac{1}{k} \one \one^T$ we see that this regime
includes the Erd\H{o}s-R\'{e}nyi graph
$\mathcal{G}\!\paren{k^t,\frac{1}{k^t}}$ at criticality.  Thus it
seems likely that in order to understand the size of the largest
component of the stochastic Kronecker graph when $c_1 = c_2 = \cdots =
c_k = 1$ it will require a deeper understanding of why the branching
process for $\mathcal{G}(n,\frac{1}{n})$ terminates with a largest
component of size
$\bigTheta{n^{\nicefrac{2}{3}}}$~\cite{Bollobas:evolutionER}.  

As a possible intermediate stage, consider a $d$-regular, connected,
non-bipartite graph $H$ on $k$ vertices and let $P$ be $\frac{1}{d}$
times the adjacency matrix of $H$.  What is the size of the largest
component in the \skg\ generated by $P$?  From a natural coupling with
$\mathcal{G}\paren{d^t,\frac{1}{d^t}}$ it is clear that it should be
at least $\bigOmega{d^{\nicefrac{2t}{3}}}$. On the other hand, since
the degree of every vertex is still asymptotically Poisson with
parameter 1, the branching process point of view would indicate that
the size of the largest component should be
$\bigTheta{k^{\nicefrac{2t}{3}}}$.  However, we note that if $H$ is
the $d$-regular graph formed by two copies of $K_{d-1}$ joined by a
perfect matching, then $H^{\otimes t}$ consists of $2^t$ copies of
$K_{(d-1)^t}$ with relatively few edges between them.  Furthermore, as
the expected degree within each of these copies of $K_{(d-1)^t}$ is
$\paren{\frac{d-1}{d}}^t \in \lilOh{1}$, the largest component in each
of these components is $\bigOh{t}$, seemingly indicating that the
overall size of the largest component is relatively small.  Thus, it
seems likely that any resolution of the case where $c_1 = c_2 = \cdots
= c_k$ will necessitate a deeper understanding of the branching
process at criticalility, and specifically, how the branching process
interacts with the underlying network of potential edges.

\end{document}